\documentclass[twoside,12pt,reqno]{amsart}
\usepackage{a4wide}
\usepackage[utf8]{inputenc}
\usepackage[T1]{fontenc}
\usepackage[english]{babel}
\usepackage{tgpagella}
\usepackage{amscd} 
\usepackage{csquotes}
\usepackage{fancyhdr}
\usepackage{multirow}
\usepackage{pgf,tikz}
\usepackage{tikz-cd,rotating}
\usetikzlibrary{arrows}
\usepackage{amsthm, amsfonts, amssymb, amsmath, latexsym, enumerate, array, color}
\usepackage[all]{xy}
\usepackage{pstricks,graphicx}
\usepackage{stmaryrd}
\usepackage{hyperref,mdwlist,mathrsfs}
\usepackage{paralist}
\usepackage{cleveref}

\newtheorem{theorem}{Theorem}[section]
\newtheorem{lemma}[theorem]{Lemma}
\newtheorem{question}[theorem]{Question}
\newtheorem{proposition}[theorem]{Proposition}
\newtheorem{corollary}[theorem]{Corollary}
 
\theoremstyle{definition}
\newtheorem{definition}[theorem]{Definition}

\newtheorem{remark}[theorem]{Remark}

\numberwithin{equation}{section}

\newcommand{\PP}{\ensuremath{\mathbb{P}}}
\DeclareMathOperator{\Soc}{Soc}

\newcommand{\fb}{\mathfrak{b}}

\title[Lefschetz Properties of complete intersesctions and  maximal variation]{ Lefschetz Properties of complete intersesctions and maximal variation of hyperplane sections}

\author[]{Mats Boij}
\address{Mats Boij: Department of Mathematics, KTH Royal Institute of Technology, 100 44 Stockholm, Sweden}
\email{boij@kth.se}

\author[]{Juan C. Migliore}
\address{Juan C. Migliore: Department of Mathematics, University of Notre  Dame, Notre Dame, IN 46556, USA}
\email{migliore.1@nd.edu}

\author[]{Rosa M.\ Mir\'o-Roig}
\address{Rosa Maria Mir\'o-Roig: Facultat de
Matem\`atiques i Inform\`atica, Universitat de Barcelona, Gran Via des les Corts Catalanes 585, 08007 Barcelona, Spain}
\email{miro@ub.edu,  ORCID 0000-0003-1375-6547}

\author[]{Uwe Nagel} 
\address{Uwe Nagel: Department of Mathematics, University of Kentucky, 715 Patterson Office Tower, Lexington, KY 40506-0027, USA}
\email{uwe.nagel@uky.edu}

\thanks{\hspace{-15pt} Boij was partially supported by the Swedish Research Council grant VR2024-04853,}
  
\thanks{\hspace{-15pt} Migliore was partially supported by Simons Foundation grant \#839618,}
  
\thanks{\hspace{-15pt}   Mir\'o-Roig was partially supported by the grant PID2024-157142NB-I00,  
}

\thanks{\hspace{-15pt} Nagel was partially supported by Simons Foundation grant \#636513}

\thanks{
This paper began as a Research in Pairs program at the CIRM in Trento in March, 2025 and was finished as a Research in Residence at the CIRM in Luminy in July 2026. 
The authors are very  grateful to both CIRMs for providing the opportunity and wonderful settings for this collaboration. }

\begin{document}

\begin{abstract} 
For an artinian complete intersection of forms of degree $d$, we prove that the multiplication by a general linear form is injective up to degree $\frac{3d-1}2$, independently of the number of variables. This is accomplished using algebraic and geometric properties of the union of two complete intersection curves and their general hyperplane sections. As a consequence, we solve, for any $d\ge 4$, the problem about the maximal variation of the hyperplane sections of a smooth hypersurface of degree $d$ posed by A. Beauville.
\end{abstract}
\maketitle

\section{Introduction} 
\vskip 4mm

A long-standing open problem in Commutative Algebra is whether all complete intersections in characteristic zero have the so-called {\it Weak Lefschetz Property (WLP)}. More precisely, let $R = k[x_0,\dots,x_n]$, where $k$ is a field of characteristic zero, and let $I = (f_0,\dots,f_n)$ be an artinian homogeneous ideal generated by a regular sequence. Then the WLP means that there exists a linear form $\ell$ for which $\times \ell : [R/I]_{t-1} \rightarrow [R/I]_t$ has maximal rank for all $t$  (such an element is called a {\it Lefschetz element}). This was conjectured by L. Reid, L. Roberts and M. Roitman to be true for an arbitrary complete intersection in 1991 in \cite{RRR}. When $\times \ell : [R/I]_{t-1} \rightarrow [R/I]_t$ is injective for a particular $t$, we say that WLP holds {\it in degree $t$}.

The maximal rank statement is known to be true for a {\it general} complete intersection with fixed generator degrees, thanks to work of R. Stanley \cite{S}, J. Watanabe \cite{W} and the above-cited \cite{RRR}. 

The most general result known on the WLP for arbitrary complete intersections is in codimension three, given by the following theorem by the second and fourth authors together with T. Harima and J. Watanabe .

\begin{theorem}[\cite{HMNW}]
When $n=2$, all complete intersections have the WLP.
\end{theorem}

From now on, we assume that the complete intersection is equigenerated, i.e. all generators have the same degree $d$. For four variables, the best result so far is the following.

\begin{theorem}[\cite{BMMN-codim4}] \label{d1=d2=d3=d4-intro} 
Let $A = R/I = R/(f_0,f_1,f_2,f_3)$, where $I$ is a complete intersection and
  $\deg f_i = d$ for all $i$ and some fixed $d \geq 1$.  Then $A$ has the WLP in all degrees $t < \lfloor\frac{3d+1}{2}\rfloor$. Furthermore, WLP holds for $d=3,4,5$.
\end{theorem}

In arbitrarily  many  variables, the best result so far for WLP for equigenerated complete intersections in codimension $\geq 5$ is the following  by Beorchia and the third author.

\begin{theorem}[\cite{B-MR}, Theorem 3.4] \label{B-MR}
Fix integers $n \geq 3$ and $d \geq 2$. Let $I = (f_0, f_1,\dots, f_n) \subset R = k[x_0,\dots,x_n]$ be a complete
intersection ideal with $\deg(f_i) = d$, $0 \leq i \leq n$.
If $d-1 < t < d + \lceil \frac{d}{n} \rceil$
then $R/I$ satisfies the WLP in degree $t$.
\end{theorem}

This is a generalization of what was previously known from the following result by Alzati and Re. 

\begin{theorem}
Let $A = \mathbb C[x_0,\dots, x_m]/I$ be an artinian algebra, with $I$ generated
by a regular sequence of forms of degree $d$. 
\begin{enumerate}
    \item (\cite{AR}, Corollary 4)
For general $z \in A_1$ the multiplication
map $\mu_z \colon A_{d-1} \longrightarrow A_d$ is injective.
\item (\cite{AR}, Theorem 1)
If $m \leq 4$ then a complete intersection of quadrics has the WLP.\end{enumerate}
\end{theorem}
    
    The main result of this paper is the following.
    
    \begin{theorem} \label{arb num var}
  Let $R/I$ be an equigenerated complete intersection in $k[x_0,\dots,x_n]$, $n \geq 4$, of forms of degree $d \geq 2$.  Then $R/I$ satisfies the WLP in degrees $t\le\frac{3d-1}2$.
\end{theorem}

This is essentially the same statement as Theorem \ref{d1=d2=d3=d4-intro}, but we note that to prove the {\it full} WLP in four variables we only need to show injectivity only up to degree $2d-2$, while in $n+1$ variables we need it up to degree $\frac{(n+1)(d-1)}{2}$. Nevertheless, it is an improvement over the other previous results. For instance, if $d=20$ and $n=10$, Theorem \ref{B-MR} gives injectivity only for $t \leq 21$ while Theorem \ref{arb num var} gives it for $t \leq 29$.

We consider two applications. A special category of equigenerated artinian complete intersections are the Jacobian ideals of smooth hypersurfaces. Ilardi \cite{I}  asked whether at least these complete intersections have the WLP. 
For smooth hypersurfaces $X = V(f)\subseteq \mathbb P^n$ the following is known about the Lefschetz properties of the Jacobian ring $R/J(X)$:
\begin{itemize}
    \item If $d\ge n+1$ then $R/J(X)$ satisfies the WLP in degree $d$  (\cite{B-P}, Corollary 4.7).
    \item If $n=3$ and $3\le d\le 6$,  then $R/J(X)$ has the WLP (\cite[Corollary 7.2]{BMMN-codim4}).
    \item If $n=3$,  then $R/J(X)$ has the WLP in all degrees degree $t<  \lfloor\frac{3d}{2}  \rfloor -1$ (\cite[Corollary 7.3]{BMMN-codim4}).
    \item
     $R/J(X)$ has the WLP in all degrees  $t<  d-1+\lceil \frac{d-1}{n}\rceil$
(See \cite[Proposition 3.1]{I} for $t=d-1$ and 
\cite[Theorem 4.3]{B-MR} for all degrees  $t\le  d-1+\frac{d-1}{n}$).
\end{itemize}
    
    As a consequence of Theorem~\ref{arb num var}, we can extend these results as follows.
    
    \begin{corollary}
    Let  $ X \subset \PP^n$ be a smooth hypersurface of degree $d\ge 2$  defined by a
homogeneous polynomial $f$. Then $R/J(X)$ has the WLP in all degrees $\leq \frac{3d}{2}-2$.
    
    \end{corollary}

As a further application, recently Beauville \cite{beauville2} applied knowledge of the WLP to the question of maximal variation of the hyperplane sections of a smooth hypersurface. (See section \ref{sect: applications} for details.) He proved:

\begin{proposition} \label{criterion} Let $X\subset \PP^n$ be a smooth hypersurface of degree $d\ge 3$. Assume $n\ge 3$, and $d\ge 4$ if $n=3$. The hyperplane sections of $X$ have maximal variation if and only if the Jacobian ideal of $X$ satisfies the WLP in degree $d$.
\end{proposition}

To date,  the  maximal variation property for the hyperplane sections of a smooth hypersurface $X\subset \PP^n$ of degree $d\ge 3$ has been established  in the following cases:

\begin{itemize}
\item If $X$ is general, with either $d\ge 3$ and $n\ge 4$, or $d\ge 4$ and $n=3$ \cite{S}, \cite{W}, \cite{beauville2}.
    \item If $d=3$ and $n\leq 4$ \cite[Theorem 1]{AR}. 
    \item If $d\geq 4$ and $n = 3$. \cite{BMMN-codim4}. 
    \item If $d \geq n+2$ \cite[Corollary 4.7]{B-MR}.
    \item if $d \geq n+3$ \cite[Theorem 1.3]{PRT}
    \item If $d = n+1$  \cite[Theorem 4.2]{B-P}.
\end{itemize}

   As an application of Theorem~\ref{arb num var}, we can extend these results and we prove:

    \begin{theorem} \label{maxvar thm}
Let \(X \subset \mathbb{P}^n\) ($n \geq 3$) be a smooth hypersurface of degree \(d \geq 4\). Then the hyperplane sections of \(X\) have maximal variation.
\end{theorem}

\section{Background}

We set $R = k[x_0,\dots,x_n]$ where $k$ is a field of characteristic zero. 

\begin{definition}
    A subscheme $X$ of $\mathbb P^n$ is \emph{Arithmetically Cohen Macaulay} (ACM) if the homogeneous coordinate ring of $X$ is Cohen-Macaulay.
\end{definition}

\begin{definition}
    For a curve $C \subset \mathbb P^n$, the {\it Hartshorne-Rao module} is the graded module
    \[
    M(C) = \bigoplus_{t \in \mathbb Z} H^1 (\mathcal I_C (t)).
    \]
    It is well known that $C$ is ACM if and only if $M(C) = 0$.
\end{definition}

\begin{definition}
    Let $A = \bigoplus_{i \geq 0} [A]_i$ be a standard graded algebra over $k = [A]_0$. Let $\ell$ be a general linear form. If $\times \ell : [A]_{i-1} \rightarrow [A]_i$ has maximal rank for all $i\ge 1$ we say that $A$ has the {\it Weak Lefschetz Property (WLP).} If, for some degree $i_0$,  $\times \ell : [A]_{i_0-1} \rightarrow [A]_{i_0}$ has (resp. does not have) maximal rank, we say that $A$ {\it has (resp. fails) the WLP in degree $i_0$.}
\end{definition}

We now recall some notation and results by Huneke and Ulrich from \cite{HU}, which we state
in our setting.

\begin{definition}
  Let $M$ be a finitely generated graded $R$-module.  We set
  \[
  a_- (M) = \min \{ i | \ [M]_i \neq 0 \}.
  \]
\end{definition}

\begin{lemma}[\cite{HU}, Socle Lemma] \label{socle lemma} Assume $k$ is algebraically closed.  Let $M$ be a
  non-zero finitely generated graded $R$-module.  Let $L \in [R]_1$
  be a general linear form and let
  \[
  0 \rightarrow K \rightarrow M(-1) \stackrel{\times
    L}{\longrightarrow} M \rightarrow D \rightarrow 0
  \]
  be exact.  If $K \neq 0$ then $a_- (K) > a_- (\Soc (D))$.
\end{lemma}

\begin{remark} \label{soc of h1} Assume the following: $S = k[x_0,x_1,\dots,x_{n-1}]$, $Z$ is a
  zero-dimensional subscheme of $\mathbb P^{n-1} = H$ with homogeneous ideal
  $I_Z \subset S$, $B = S/I_Z$, $\bar B$ is an artinian reduction of
  $B$ by a linear form, and
  \[
  0 \rightarrow \bigoplus_{i=1}^{b_{n-1}} S(-n_{n-1,i}) \rightarrow \dots\rightarrow
  \bigoplus_{i=1}^{b_1} S(-n_{1,i}) \rightarrow S \rightarrow B
  \rightarrow 0
  \]
  is the minimal free resolution.  In \cite{HU} it is also pointed out
  that then
  \[
  \Soc (H^1_\ast(\mathcal I_Z)) = \bigoplus_{i=1}^{b_{n-1}} k(-n_{n-1,i} +n) =
  \Soc (\bar B)(1).
\]
\end{remark}

Still following the work of \cite{HU}, 
let $M = H^1_\ast (\mathcal I_{C})$, where $C \subset \PP^n$ is a curve whose general hyperplane section is $Z = C \cap H$.
From (\ref{std exact}) below we see that $\Soc (D) \subset
\Soc (H^1_\ast(\mathcal I_{Z|H}))$.  We thus obtain
\[
a_- (K) > a_- (\Soc (D)) \geq a_- (\Soc (H^1_\ast (\mathcal I_{Z|H})))
\geq a_- (\Soc (\bar B)) -1;
\]
that is
\begin{equation} \label{hu ineq} a_-(K) \geq a_- (\Soc (\bar B)).
\end{equation}

\begin{remark} \label{trans socle lemma}
Notice that $K$ represents the forms in $I_{Z|H}$ that do not lift to $I_C$. Thus one way of phrasing the result of the Socle Lemma, which we will use, is that

\begin{quotation}

{\it the socle of $\bar B$ starts no later than the least degree of a form vanishing on $Z$ that does not lift to $C$.}

\end{quotation}

\noindent Since the latter degree can sometimes be read from the Hilbert function of $\bar B$, this can be used to force socle elements in $\bar B$.
\end{remark}

The notion of basic double linkage will be used in an important way in this paper. As its name suggests, it was created with applications to liaison theory in mind. It was introduced by Lazarsfeld and Rao \cite{LR} in 1983 as a key tool in the structure theorem that they proved for special liaison classes of curves in $\mathbb P^3$. In the form used in this paper, it was generalized by Bolondi and Migliore in 1989. It has been substantially generalized further in ways that we omit here (see for instance \cite{KMMNP}). We state it here only in the special situation that we will need.

\begin{definition}[\cite{BM4} Definition 1.7] \label{BDL def}
    Let $Z \subset \mathbb P^m$ be a finite set of points and choose forms $f_1,\dots,f_{m-1} \in I_Z$ such that $(f_1,\dots,f_{m-1})$ is a regular sequence. Let $g \in R$ be a homogeneous polynomial so that $(g,f_1,\dots,f_{m-1})$ is also a regular sequence. Then the ideal $g \cdot I_Z + (f_1,\dots,f_{m-1})$ is the saturated ideal of a zero-dimensional scheme $Z'$ in $\mathbb P^m$. If $Y$ is the complete intersection scheme defined by $(g,f_1,\dots,f_{m-1})$ then as schemes we have $Z' = Z \cup Y$. If moreover $g$ does not vanish at any point of $Z$ then the latter is a disjoint union. We say that $Z'$ is a {\it basic double link} of $Z$, and that $g \cdot I_Z + (f_1,\dots,f_{m-1})$ is a {\it basic double link} of $I_Z$.
\end{definition}

\begin{remark}
    We stress that the ideal constructed in Definition \ref{BDL def} is saturated, as this will be important in our proof. We also remark that $Z$ and $Z'$ are linked in two steps (hence the name), although we will not need this here.
\end{remark}

\section{The basic construction}

We introduce the set-up for establishing our main result by generalizing the construction used in \cite{BMMN-codim4} from the case $n=3$ to the situation where $n\ge 3$. Then we prove some initial properties and recall needed results. 

Our construction is still based on a union of two complete intersection curves. 
Let $I = (f_0,f_1,\dots,f_n)$ be an artinian complete intersection in $R = k[x_0,\dots,x_n]$ of forms of degree $d$, and consider  
$R/I = R/(f_0,\dots,f_n)$. Let $H$ be a general hyperplane defined by a linear form $\ell$ and let $\overline{R}=R/(\ell)$ be the coordinate ring of $H\cong\mathbb P^{n-1}$. We will denote the image in $\overline{R}$ of a form $f$ in $R$ by $\overline f$ and we note that $\overline{R} \cong k[\overline{x}_0,\overline{x}_1,\dots,\overline{x}_{n-1}]$. We will need the following ideals and varieties defined by them.
\[
\begin{array}{rclllccccccccccccc}
I_{X_0} & = & (f_4,\dots,f_{n}) & & (\dim X_0 = 3) \\
I_{V_1} & = & (f_{0},f_{1}) && (\dim V_1 = n-2) \\
I_{V_2} & = & (f_{2},f_{3}) && (\dim V_2 = n-2) \\
I_V & = & I_{V_1} \cap I_{V_2} && (\dim V = n-2) \\
V & = & V_1 \cup V_2 \\
I_Y & = & (f_{0}, f_{1}, f_{2}, f_{3}) && (\dim Y = n-4) \\
& = & I_{V_1} + I_{V_2} \\
Y & = & V_1 \cap V_2 \\
I_{C_1} & = & (f_0, f_1,f_{4},\dots, f_{n}) & & (\dim C_1 = 1) \\
& = & I_{V_1} + I_{X_0} \\
C_1 & = & V_1 \cap X_0 \\
I_{C_2} & = & (f_2, f_3, f_4,\dots, f_{n}) & & (\dim C_2 = 1) \\
& = & I_{V_2} + I_{X_0} \\
C_2 & = & V_2 \cap X_0 \\
I_C & = & I_{C_1} \cap I_{C_2} && (\dim C = 1) \\
C & = & V \cap X_0 \\
C_1 \cap C_2 & = & \emptyset
\end{array}
\]

Assume that the forms $f_0,\dots,f_n$ form a generically chosen basis for the $(n+1)$-dimen\-sional vector space  $[I]_d$. By applying Bertini's theorem (see for instance \cite{kleiman}) several times, we can assume that the varieties  
\[
X_0, V_1, V_2, Y, C_1, C_2
\]
are all smooth and irreducible, as well as complete intersections. $C$ is smooth but not irreducible. Note also that
\[
I_V = I_{V_1} \cap I_{V_2} = I_{V_1} \cdot I_{V_2}
\]
since $V_1$ and $V_2$ are ACM of codimension two and they meet in the expected codimension four in $\PP^n$).

For the smooth curve $C = C_1 \cup C_2$ defined above, we define 
\[
Z_1 = C_1 \cap H, \ Z_2 = C_2 \cap H,\ \ \text{and} \ \ S_0 = X_0 \cap H.
\]
Since $C_1$, $C_2$ and $X_0$ are complete intersections, we have the ideals
\[
\begin{array}{rcl}
I_{S_0}  = I_{X_0}\overline{R}& = & (\overline{f_4}, \overline{f}_5,\dots, \overline{f}_{n}), \\
I_{Z_1} = I_{C_1}\overline{R}& = & (\overline{f}_0, \overline{f}_1,\overline{f}_4,\dots, \overline{f}_{n}) =(\overline{f}_0, \overline{f}_1) + I_{S_0}, \\
I_{Z_2} = I_{C_2}\overline{R}& = & (\overline{f}_2, \overline{f}_3,\overline{f}_4,\dots, \overline{f}_{n})=(\overline{f}_2, \overline{f}_3) + I_{S_0}.
\end{array}
\]
However, the ideal $I_Z\subseteq \overline{R}$, defining $Z$ is not given by $I_C\overline{R}$ since $C$ is not ACM. The crucial point of our method is to understand the difference between these ideals.

\begin{lemma} \label{facts about S,C}
\begin{itemize}
\item[(a)] $H^i(\mathcal I_V(t)) = 0$ for all integers $t$ and $1 \leq i \leq n-3$.

\item[(b)]  $H^{n-2}(\mathcal I_V(t)) \cong H^{n-3}(\mathcal I_Y(t))$ for all $t$.

\item[(c)] $I_C = I_V + (f_4,\dots,f_{n}) = I_{V_1} \cdot I_{V_2} + (f_4,f_5,\dots,f_{n})$. In particular, the ideals are saturated.

\item[(d)] We have that
  \[
  M(C) \cong A = R/ ( f_0,f_1,f_2,f_3, ,\dots,f_n ).
  \]
In particular, the minimal free resolution of $M(C)$ is given by the
  Koszul resolution.
\end{itemize}
\end{lemma}

\begin{proof}
For (a), consider the exact sequence
\[
0 \rightarrow I_V \rightarrow I_{V_1} \bigoplus I_{V_2} \rightarrow I_{V_1} + I_{V_2} \rightarrow 0.
\]
Since $I_{V_1} + I_{V_2} = I_Y$ is a complete intersection of codimension four, we have $h^i_*(\mathcal I_Y) = 0$ for $1 \leq i \leq n-4$. Furthermore, setting $h^i_*(\mathcal I_X) := \dim \bigoplus_{t \in \mathbb Z} H^i (\PP^n, \mathcal I_X (t))$ for a scheme $X \subset \PP^n$, 
we have $h^i_*(\mathcal I_{V_1}) = 0$ for $1 \leq i \leq n-2$, and same for $V_2$.
Since $I_{V_1}$ and $I_{V_2}$ define ACM schemes and  $I_{V_1} + I_{V_2}$ is the saturated ideal of the complete intersection $Y$ of dimension $n-4$, sheafifying and taking the long exact sequence in cohomology gives the result. The same sequence gives (b).

We turn to (c). We know $V$ has dimension $n-2$ (so $R/I_V$ has Krull dimension $n-1$) and $R/I_V$ has depth at least $n-2$  by (a). Since $(f_4,\dots,f_{n})$ is generated by a regular sequence, this proves the first part of (c). The second part then follows from the equality $I_{V_1} \cap I_{V_2} = I_{V_1} I_{V_2}$ mentioned above.

For (d), the first part is clear. For the Hartshorne-Rao module, it follows from the exact sequence
\[
0 \rightarrow I_C \rightarrow I_{C_1} \oplus I_{C_2} \rightarrow I_{C_1} + I_{C_2} \rightarrow 0,
\]
sheafifying and taking cohomology (using that $C_1$ and $C_2$ are disjoint and ACM).
\end{proof}

\begin{remark} \label{deggens}
Part (c) of Lemma \ref{facts about S,C} gives that $I_C$ is minimally generated by $n-3$ forms of degree $d$ (which generate $I_{S_0}$) and four forms of degree $2d$ (which generate $I_V=I_{V_1}I_{V_2}$).
\end{remark}

Let $\ell$ be a general linear form and let $H$ be the hyperplane
defined by $\ell$.  Let $Z$ be the zero-dimensional scheme cut out on
$C$ by $H$.  As a subscheme in $\mathbb P^n$, $Z$ has a homogeneous
ideal that we will denote $I_Z$, and as a subscheme of $H$ it has a
homogeneous ideal $I_{Z|H}$.  Consider the exact sequence of sheaves
\[
0 \rightarrow \mathcal I_C(t-1) \stackrel{\times L}{\longrightarrow} \mathcal I_C (t) \rightarrow \mathcal I_{Z|H}(t) \rightarrow 0
\]
which yields the long exact sequence
{\footnotesize
  \begin{equation} \label{std exact} 0 \rightarrow [I_C]_{t-1}
    \stackrel{\times L}{\longrightarrow} [I_C]_t \rightarrow
    [I_{Z|H}]_t \rightarrow [A]_{t-1} \stackrel{\times
      L}{\longrightarrow} [A]_t \rightarrow H^1(\mathcal
    I_{Z|H}(t)) \rightarrow H^2(\mathcal I_C(t-1)) \rightarrow
    H^2(\mathcal I_C(t)) \rightarrow 0.
  \end{equation}}

  \begin{remark}
Since the Hilbert function of $C$ is determined, the question of
whether $A$ has the WLP depends completely on understanding the
Hilbert function of $I_{Z|H}$, and in particular on when $I_Z$ has a generator that does not lift to $I_C$. Since $Z$ will always be contained  in the hyperplane $H$, henceforth by $I_Z$ we will mean the ideal of $Z$ in $H$. Above we maintained $I_{Z|H}$ for emphasis since we view $C$ in $\mathbb P^n$ but $Z$ in $H$. 
\end{remark}

Since $H$ is a general hyperplane, the hyperplane sections $Z_1$ and $Z_2$ are two complete intersections of type $(d,d, \dots, d)$ (defining a complete intersection set of points   of degree $d^{n-1}$ in $H = \PP^{n-1}$) lying on the same smooth surface $S_0$ (again using Bertini).

We want to know under certain circumstances that any elements of $I_Z$ in certain degrees automatically lift to $I_C$. This will lead to a contradiction. 

\begin{remark} \label{irred curve}
If the ideal of the union $Z = Z_1 \cup Z_2$ has a minimal generator $g$ of degree $e \ge d$, it cuts out on the surface $S_0$ a curve containing $Z$, and it would be natural, and very useful, to assume without loss of generality that this curve is smooth and irreducible. Unfortunately, we have not been able to prove this.
\end{remark}

\begin{remark}
The general hyperplane section of an irreducible curve has the so-called {\it Uniform Position Property (UPP)} \cite{harris}, meaning that any two subsets of the same cardinality have the same Hilbert function. Thus $Z_1$ alone or $Z_2$ alone have the UPP. However, it is easy to see that the union, $Z$, does not have the UPP. Nevertheless, $Z$ does have useful special properties. In \cite{BMMN-codim4}, a key method is the {\it Symmetry Principle}. It was introduced for the case $n=3$, but it applies equally well to our more general situation. It says that 

\begin{quotation} 
{\it there can be no geometric or numerical property distinguishing  $Z_1$ from $Z_2$.}
\end{quotation}

\noindent Despite not having the UPP, we will see in Lemma \ref{lgp rmk} that the  points of $Z $ are in {\it linear general position (LGP)}, meaning that no $n$ of the points lie on a hyperplane in $H = \mathbb P^{n-1}$. The argument used in \cite{harris} to prove UPP was based on proving that the monodromy group associated to a general hyperplane section of an irreducible curve is the full symmetric group. In our situation we have the following slightly weaker result. 
\end{remark}

\begin{lemma}[\cite{POLITUS}] \label{lgp rmk}
    Let $C$ be the union  of two nondegenerate curves $C_1$ and $C_2$ in $\PP^n$, both  of which are reduced and irreducible, of degrees $d_1, d_2$ respectively. We do not assume that the curves have the same degree or that they are disjoint. Let $Z = Z_1 \cup Z_2$, where $Z_i = C_i \cap H$ and $H$ is a general hyperplane. 

    \begin{itemize}
    \item[\rm (i)] If $G_i$ is the monodromy group of $Z_i$ as $H$ varies, $i = 1,2$, then the monodromy group of $Z$ is $G_1 \times G_2$.  More precisely, $G_1 \times G_2 = S_{d_1} \times S_{d_2}$, the direct product of the symmetric groups. 
    
        \item[\rm (ii)]  Both $Z_1$ and $Z_2$ have the UPP separately and the union has LGP.
        
        \item[\rm (iii)]  Suppose that $Z$ lies on a unique curve $Y$ of a certain degree and suppose that $Y = Y_1 \cup Y_2$. Then both $Y_1$ and $Y_2$ are irreducible and (up to interchanging $C_1$ and $C_2$) $Z_i \subset Y_i$ for $i = 1,2$.
        \end{itemize}
    \end{lemma}

\begin{proof}
Part (i)  is immediate from Exercise 11.4 of  \cite{EHbook}. Its proof and that of  part (ii) are modifications of resutls in \cite[Lemma, page 698]{HarrisDuke},  \cite[Uniform Position Lemma, page 197]{harris}, \cite[Remark 1.7]{Rathmann},  \cite[pages 111--113]{ACGH} and especially  \cite[Chapter 11]{EHbook}. Part (iii) follows from the fact that both $\Gamma_1$ by itself and $\Gamma_2$ by itself have the UPP, and from the uniqueness of $Y_1$ and $Y_2$. See \cite{POLITUS} for details.
\end{proof}

\section{Main Theorem}

We first recall a well known fact, for which we include a short proof for convenience.

\begin{lemma}\label{lemma:extension}
    Let $A$ be an artinian graded $k$-algebra and let $B=K\otimes_k A$, where $K$ is a field extension of $k$ and $k$ is an infinite field. For any fixed integer $t \ge 0$, the following conditions are equivalent: 

\begin{itemize}
\item[(i)] The multiplication $\times \ell\colon [A]_{t-1} \to  [A]_{t}$ has maximal rank for some $\ell$ in $[A]_1$.
\item[(ii)] The multiplication $\times \ell'\colon [B]_{t-1} \to  [B]_{t}$ has maximal rank for some $\ell'$ in $[B]_1$.
\end{itemize}
\end{lemma}

\begin{proof} 
    It is clear that (i) implies (ii) and we need to prove the converse. Assume that $A = k[x_1,x_2,\dots,x_n]/I$ for some homogeneous ideal. If we fix monomial bases for $[A]_{t-1}$ and $[A]_{t}$, these monomials will also form bases for $[B]_{t-1}$ and $[B]_{t}$. The matrix $M_i$ for multiplication $\times x_i\colon[A]_{t-1}\to[A]_t$ is identical to the matrix for multiplication $\times x_i\colon[B]_{t-1}\to[B]_t$, for $i=1,2,\dots,n$. If condition (ii) holds, there are coefficients $a_1,a_2,\dots,a_n$ in $K$ such that $M = \sum_{i=1}^n a_i M_i$ has maximal rank, meaning that there is at least one non-zero maximal minor of $M$. Hence this minor is non-zero as a polynomial in the coefficients. Since this polynomial is defined over $k$, and $k$ is infinite, we can find coefficients $b_1,b_2,\dots,b_2$ in $k$ such that $\sum_{i=1}^n b_i M_i$ has maximal rank, showing that (i) holds.
\end{proof}

We are now in a situation where we can prove the main result. 

\begin{theorem} 
  Let $R/I$ be an equigenerated complete intersection in $k[x_0,\dots,x_n]$, $n \geq 4$, of forms of degree $d \geq 2$.  Then $R/I$ satisfies the WLP in degrees $t\le\frac{3d-1}2$.
\end{theorem}

\begin{proof} Thanks to Lemma \ref{lemma:extension}, from now on we can and will assume that $k$ is algebraically closed in order to  apply the Socle Lemma.
We already know that $A=R/I$ satisfies WLP in degrees $t\le d$ by \cite[Corollary 4]{AR}. The socle degree of $A$ is $(n+1)(d-1)$ so the range we are considering is when WLP means injectivity of multiplication on $A$. 
Assume by contradiction that $\times \ell:A_{d+k-1}\longrightarrow A_{d+k}$ is not injective and that $k$ is the least integer where this holds. By the long exact sequence \eqref{std exact} and by Remark~\ref{deggens}, this implies that there is an element $g$ in $[I_Z]_{d+k}$ which  does not lift to $I_C$ and so is not in $[I_{S_0}]_{d+k}$. Since $I_Z = I_{Z_1}\cap I_{Z_2}$, this means that we can write 
\begin{equation} \label{eqn sec 4}
g = p_0\overline{f}_0+p_1\overline{f}_1 + \sum_{i=4}^n p_i\overline{f}_i = 
p_2\overline{f}_2+p_3\overline{f}_3 + \sum_{i=4}^n p_i'\overline{f}_i.
\end{equation}
Now assume that $g'$ is an element of degree $e$ in $I_Z\setminus I_{S_0}$, where $d+k\le e<2d-k$.  Then we can write 
\[
g' = q_0\overline{f}_0+q_1\overline{f}_1 + \sum_{i=4}^n q_i\overline{f}_i = 
q_2\overline{f}_2+q_3\overline{f}_3 + \sum_{i=4}^n q_i'\overline{f}_i.
\]
We have that 
\begin{equation}\label{eq:1}
    p_0g'-q_0g = (p_0q_1-q_0p_1)\overline f_1 + \sum_{i=4}^n (p_0q_i-q_0p_i)\overline f_i
\end{equation}
which shows that $(p_0q_1-q_0p_1)\overline f_1$ belongs to $I_Z$. Since we can assume that $\overline f_1$ does not vanish on any point of $Z_2$, we conclude that $p_0q_1-q_0p_1$ vanishes on $Z_2$. Since the degree of $p_0q_1-q_0p_1$ is $k+e-d<d$ and $I_{Z_2}$ is generated in degree $d$ this forces that $p_0q_1-q_0p_1=0$. Hence we get from Equation~\ref{eq:1} that 
\[
p_0g' = q_0g + \sum_{i=4}^n (p_0q_i-q_0p_i)\overline f_i \in (g,\overline f_4,\dots,\overline f_n).
\]
In the same way we get that 
\[
p_1g' = q_1g + \sum_{i=4}^n (p_1q_i-q_1p_i)\overline f_i \in (g,\overline f_4,\dots,\overline f_n).
\]
If a general element in the vector space generated by $p_0$ and $p_1$ vanishes on a point of $Z_1$, it has to vanish on all points of $Z_1$ since the monodromy group is the full symmetric group. The same is true for $Z_2$ and we can assume that $p_0$  is non-zero on every point of $Z$. Hence also $\operatorname{gcd}(p_0,p_1)$  is non-zero on every point of $Z$.  This means that 
\[
\frac{p_0\overline f_0+p_1\overline f_1}{\operatorname{gcd}(p_0,p_1)} = \frac{g - \sum_{i=4}^n p_i\overline{f}_i }{\operatorname{gcd}(p_0,p_1)}
\]
is in the ideal $I_Z$. By the minimality of the degree of $g$, we conclude that $\operatorname{gcd}(p_0,p_1)=1$. Since we proved that $p_0q_1-q_0p_1=0$, we get that $q_0$ is divisible by $p_0$ and we conclude that 
\[
g' = \frac{p_0g'}{p_0} = \frac{q_0}{p_0}g + \sum_{i=4}^n \left(q_i-\frac{q_0}{p_0}p_i\right)\overline{f}_i \in (g,\overline f_4,\dots,\overline{f}_n). 
\]
We have now proven that 
\[
[I_Z]_e = [(g,\overline{f}_4,\dots,\overline{f}_n)]_e \quad \text{for $e<2d-k$}
\]
The reduction of $\overline{R}/(g,\overline{f}_4,\dots,\overline{f}_n)$ by a non-zero divisor does not have any socle. Hence the artinian reduction of $\overline R/I_Z$ cannot have socle in degree less than $2d-k-1$. By Remark~\ref{trans socle lemma}, $g$ lifts to $I_C$ if $d+k<2d-k-1$, which is equivalent to $2k\le d-2$ since $d$ and $k$ are integers, and hence equivalent also to $d+k\le \frac{3d}2-1$. This gives a contradiction to the initial assumption that $g$ does not lift to $I_C$ and we conclude that the multiplication by $\ell$ is injective from $A_{j-1}$ to $A_j$ when $j\le \frac{3d}2-1$ and WLP holds in this range.
This concludes the case when $d$ is even and the cases where $d$ is odd and $k< (d-1)/2$.

What remains is the case where $d$ is odd and $k=(d-1)/2$. Above we showed that every element $g'$ of $I_Z$ in degree $\leq 2d-k-1$ is in the ideal $(g,\overline{f}_4,\dots,\overline{f}_n)$.  With our current assumption note that $2d-k-1=d+k$. We conclude that there is a unique minimal generator $g  \in I_Z \subset I_{Z_1}$ in degree $d+k$. 
As in (\ref{eqn sec 4}) we write
\[
g   = p_0\overline{f}_0+p_1\overline{f}_1 + \sum_{i=4}^n p_i\overline{f}_i .
\]

\vspace{.1in}

\noindent \underline{Claim 1}:  {\it The ideal $(\overline{f}_4,\dots,\overline{f}_n,g)$ defines a complete intersection curve, and $(\overline{f}_4,\dots,\overline{f}_n,g) \subset I_{Z_2}$.}

\vspace{.1in}

\noindent The ideal $(\overline{f}_4,\dots,\overline{f}_n)$ defines a smooth irreducible surface and $g$ does not vanish completely on this surface. Thus the ideal defines a curve, and counting the number of generators we get that it is a complete intersection. This completes the proof of  Claim 1.

\vspace{.1in}

\noindent \underline{Claim 2}:  {\it The polynomial $f_1 \in I_{Z_1}$ is a regular element on  $\overline{R}/ ( \overline{f}_4,\dots,\overline{f}_n,g)$.  Hence the ideal
\[
J :=\overline{f}_1I_{Z_2} + (g,\overline{f}_4,\overline{f}_5,\dots,\overline{f}_n)
\]
is a basic double link of $I_{Z_2}$.}

\vspace{.1in}

\noindent Using the inclusion in Claim 1, the conclusion about the basic double link follows once we show that $f_1$ is a regular element. But $I_{Z_1}$ is generated in degree $d$, so the base locus of $[I_{Z_1}]_d$ is precisely  $Z_1$  scheme-theoretically. Since $f_1$ is general in $[I_{Z_2}]_d$, and the ideal is Cohen-Macaulay, Claim 2 follows.

So $J$ is a saturated ideal defining the union of $Z_2$ and the  complete intersection scheme defined by $(\overline{f}_1, \overline{f}_4, \dots. \overline{f}_n, g)$. 

\vspace{.1in}

\noindent \underline{Claim 3}: {\it The ideal $J$ is the saturated ideal of the disjoint union of three complete intersections: $Z_1, Z_2$ and a complete intersection scheme $Y$ of type $(d^{n-2},k)$, specifically $I_Y = (\overline{f}_1, \overline{f}_4, \dots. \overline{f}_n, p_0)$. }

\vspace{.1in}

The fact that $Z_1$ and $Z_2$ are disjoint has been discussed already. The fact that $J \subset I_{Z_1} \cap I_{Z_2}$ is also clear from the construction. The scheme defined by  $(\overline{f}_1, \overline{f}_4, \dots. \overline{f}_n, g)$ is a complete intersection that is disjoint from $Z_2$ by the generality of $f_1$.

Since  $(\overline{f}_1, \overline{f}_4, \dots. \overline{f}_n, g) \subset I_{Z_1}$, it provides a link of $Z_1$ with a residual scheme $Y$. We will show that $Z_1$ and $Y$ have no components in common (i.e. it is a geometric link) and that $Y$ is a complete intersection of type $(d^{n-2},k)$.
Specifically,  the link is
\[
(\overline{f}_1, \overline{f}_4, \dots. \overline{f}_n, g) : I_{Z_1} .
\]
It follows from a standard computation in liaison theory (e.g. using mapping cones) that a complete intersection of type $(d^{n-1})$ is linked inside a complete intersection of type $(d^{n-2},d+k)$ to a complete intersection of type $(d^{n-2}, k)$ (see, e.g., \cite{JM-book}).
The conclusion that the defining ideal of $Y$ is as claimed then follows:  from the above ideal quotient, together with (\ref{eqn sec 4}), we get that the defining ideal of $Y$ is $I_Y = (p_0,\overline{f}_1,\overline{f}_4,\dots,\overline{f}_n)$. 

Finally, we have to show that the above link is geometric. But we saw above  that $p_0$  is non-zero on every point of $Z$, so it is a geometric link. In other words, we have that
\[
   J =  \overline{f}_1I_{Z_2} + (\overline{f}_4,\overline{f}_5,\dots,\overline{f}_n,g) = I_{Z} \cap (p_0,\overline{f}_1,\overline{f}_4,\dots,\overline{f}_n)
\]
defines a disjoint union. We  have thus shown Claim 3

\vskip 2mm
Now we focus on the generators of $I_Z$ of degree $d+k+1$.
Assume that these generators are $g_1',g_2',\dots,g_r'$. Since $g$ does not lift, we know that  the artinian reduction of $\overline{R}/I_Z$ must have socle in degree $d+k$ (see Remark \ref{trans socle lemma}), so $r\ge 1$. Note that for each $i$, $g_i' \in I_{Z_1}$ so for all $i=1,2,\dots,r$  we
can write
\[
  g_i' = q_{i,0}\overline{f}_0+q_{i,1}\overline{f}_1 + \sum_{j=4}^n q_{i,j}\overline{f}_j,   
\]
Our argument will focus on the number, $r$, of these generators.

\vspace{.1in}

\noindent \underline{Claim 4}: {\it $r \leq 2$.}

\vspace{.1in}

Note that
 \begin{equation} \label{inclusion}
p_0(g_1',g_2',\dots,g_r')  + (\overline{f}_4,\overline{f}_5,\dots,\overline{f}_n,g)  \subseteq J.
 \end{equation}
 Both sides are homogeneous ideals, and we focus on the degree $2d$ component, recalling that $\deg p_0 = k = \frac{d-1}{2}$ and $\deg p_0 g_i' = 2d$. 
 Begin with the quotient $J/(\overline{f}_4,\overline{f}_5,\dots,\overline{f}_n,g)$. Observe that this quotient has at most two generators in degree $2d$ corresponding to the generators of $\overline{f}_1I_{Z_2} = \overline{f}_1(\overline{f}_2,\overline{f}_3)$.
 
The smaller ideal in (\ref{inclusion})  is the sum of two subideals. How big can the intersection be? If some scalar linear combination $\sum_{i=1}^r \alpha_i p_0g_i'$ were in $(\overline{f}_4,\overline{f}_5,\dots,\overline{f}_n,g)$, we would get that  $\sum_{i=1}^r \alpha_i g_i'$ was in $(\overline{f}_4,\overline{f}_5,\dots,\overline{f}_n,g)$, since $I_Z:p_0 = I_Z$ and $(\overline{f}_4,\overline{f}_5,\dots,\overline{f}_n,g)\subseteq I_Z$. This contradicts the assumption that $g_1',g_2',\dots,g_r'$ are minimal generators. 
We conclude that in degree $2d$ we have $(\overline{f}_4,\overline{f}_5,\dots,\overline{f}_n,g) \cap p_0(g_1',g_2',\dots,g_r')  = 0$.

The first observation shows that in order to fill $[J]_{2d}$ the additional contribution from $p_0 (g_1',\dots,g_r')$  can increase the dimension by at most 2. The second observation says that there is no overlap between the two subideals in degree $2d$. 
Hence we get that $r\le 2$.  This proves Claim 4.

\vspace{.1in}

\noindent \underline{Claim 5}: {\it $r= 2$ is impossible.}

\vspace{.1in}

Now we will zero in on the residual scheme $Y$. First, suppose $r=2$.  Then
 \[
    p_0(g_1',g_2')  + (\overline{f}_4,\overline{f}_5,\dots,\overline{f}_n,g) \subseteq J.
  \]
But  both ideals are generated in degrees at most $2d$ and they agree up to this degree, so we have equality: $p_0(g_1',g_2')  + (\overline{f}_4,\overline{f}_5,\dots,\overline{f}_n,g) = J.$.
Then 
 \[
 J + (p_0)  =      (\overline{f}_4,\overline{f}_5,\dots,\overline{f}_n,g) + p_0(g_1',g_2') + (p_0) =(p_0,\overline{f}_4,\overline{f}_5,\dots,\overline{f}_n,g) .
\]
Since $J$ has codimension $n-1$ in $H = {\mathbb P}^{n-1}$, the codimension of the ideal on the right-hand side must be at least $n-1$. Thus, it is a complete intersection and, in particular, a saturated ideal.  Note that $p_0$ vanishes on $Y$ and $J \subset I_Y$ so the ideal on the left is not artinian. In fact, we have seen that $p_0$ does not vanish at any point of $Z$, so this is the saturated ideal of $Y$.

In Claim 3 we showed that $I_Y = (\overline{f}_1, \overline{f}_4, \dots. \overline{f}_n, p_0)$. But their generators have different degrees, since $g$ has degree $d+k$ while $\overline{f}_1$ has degree $d$. We get a contradiction to the assumption that $g$ does not lift to the curve and we conclude that we have injectivity also in this degree.  We have proven Claim 5.

\vspace{.1in}

\noindent \underline{Claim 6}: {\it $r= 1$ is impossible.}

\vspace{.1in}

Since we assume that $g$ does not lift, there has to be socle in degree $d+k$ in the artinian reduction, $A$, of $\overline{R}/I_Z$ (see Remark \ref{trans socle lemma}). Let $m$ denote the number of first syzygies of the form $\ell_i g' + q_i g + \sum_{j=4}^n c_{i,j} \overline{f}_j$. Note every syzygy gives an equation
\begin{equation} \label{eqn socle}
q_i g + \sum_{j=4}^n c_{i,j} \overline{f}_j = - \ell_i g'. 
\end{equation}

The linear forms $\ell_1,\ell_2,\dots,\ell_m$ need to be linearly independent since $(g,\overline{f}_4,\dots,\overline{f}_n)$ is a complete intersection. In the artinian reduction, $\widetilde{R}$, of $\overline{R}/I_Z$,  the generators of degrees $d$ and $d+k$ generate a complete intersection of Krull dimension $1$, which necessarily has no socle. We consider three cases for $m$.

\begin{itemize}

\item If $m\ge n-1$, the $m$ syzygies imply that the generator of degree $d+k+1$ is a socle element in the   complete intersection $\overline{R}/(\overline{f}_4,\dots, \overline{f}_n, g)$, giving an immediate contradiction.  Hence we conclude that $m\le n-2$.

\vspace{.1in}

\item If $m=n-2$, (\ref{eqn socle}) gives that
\[
(\overline{f}_4,\overline{f}_5,\dots,\overline{f}_n,g) : g' \supseteq (\ell_1,\ldots,\ell_{n-2}). 
\]
Since $(\overline{f}_4,\overline{f}_5,\dots,\overline{f}_n,g)$ is a complete intersection, the ideal on the left-hand side is unmixed of codimension $n-2$. The ideal on the right-hand side is prime of codimension $n-2$. Thus, both ideals must be equal. 
Setting $\fb = (\overline{f}_4,\dots,\overline{f}_n, g, g')$, multiplication by $g'$ gives an exact sequence
\[
0 \to \overline{R}/(\ell_1,\ldots,\ell_{n-2}) (- d-k-1) \to \overline{R}/(\overline{f}_4,\overline{f}_5,\dots,\overline{f}_n,g) \to \overline{R}/\fb \to 0.
\]
It yields
\begin{align*}
  \label{eq:compare deg} 
\deg \fb = -1 + \deg (\overline{f}_4,\dots,\overline{f}_n, g).  
\end{align*}

Let us assume that the artinian reduction $\widetilde{R}$ is given by the quotient of $\overline{R}/I_Z$ by a general linear form $\ell' \in \overline{R}$. Denote the image of $g', \overline{f}_i$  in $\widetilde{R}$ by $\widetilde{g'}, \widetilde{f}_i$  and similarly for other elements.  
The above equality implies  that that $\widetilde{g'}$ passes through a subscheme of the complete intersection scheme defined by 
$(\widetilde{f}_4,\dots,\widetilde{f}_n, \widetilde{g})$ and that subscheme has degree $-1 + \deg (\widetilde{f}_4,\dots,\widetilde{f}_n, \widetilde{g})$. 
Hence, by the Cayley?Bacharach theorem (see \cite[Theorem 1.1]{K}), $\widetilde{g'}$ has to be in that complete intersection ideal.
But this forces $g' \in (\overline{f}_4,\dots,\overline{f}_n, g)$,  contradicting the assumption that $g'$ is a minimal generator of $I_Z$. Thus, $m = n-2$ is impossible.  

\vspace{.1in}

\item Assume $m<n-2$. 
The linear part of the resolution of the $m$ first syzygies is a subcomplex of the Koszul complex on the linear forms $\ell_1,\ell_2,\dots,\ell_m$. If $m<n-2$, this Koszul complex does not reach homological degree $n-1$, so there cannot be any socle in degree $d+k$ in the artinan reduction. 

\end{itemize}

We conclude that $r=1$ is also impossible, completing the proof of Claim 6 and of our theorem.  
\end{proof}

\section{Applications} \label{sect: applications}

 This last section is entirely devoted to applications of our main result (Theorem \ref{arb num var})  For this section we assume further that the field $k$ is algebraically closed.

 \subsection{Jacobian ideals.} 
Let  $ X \subset \PP^n$ be a smooth hypersurface of degree $d$  defined by a
homogeneous polynomial $f$. The Jacobian ideal $J(X)$ is the homogeneous Artinian 
ideal generated by the partial derivatives $\partial f/\partial x_i$, $0 \le  i \le  n$. Since $X$ is smooth, $J(X)$ is a complete intersection Artinian ideal 
generated by $n+1$ forms of degree $d-1$.
In \cite{I}, G. Ilardi posed the following question:

\begin{question} 
Does the Jacobian ring $R/J(X)$ of a smooth hypersurface have the WLP?
\end{question} 

So far, as mentioned in the introduction, very little is known.

As a consequence of  Theorem \ref{arb num var} we  get an improvement of Beorchia and Mir\'o-Roig's result in the case of arbitrary number of 
variables \cite[Theorem 4.3]{B-MR} :

\begin{corollary}
 Let  $ X \subset \PP^n$ be a smooth hypersurface of degree $d\ge 2$  defined by a
homogeneous polynomial $f$. The Jacobian ring $R/J(X)$ has the WLP in all degrees $t\le\frac{3d}2-2$.    
\end{corollary}

\subsection{Maximal variation of hyperplane sections.}  

Following ideas recently developed by Beauville in \cite{beauville} and \cite{beauville2}  we  apply our results on the WLP for equigenerated Artinian complete intersection ideals to the study of the maximal variation problem for the hyperplane sections of a smooth hypersurface $X\subset \PP^n$, $n\ge 3$, of degree $d\ge 2$.

Understanding how hyperplane sections vary in moduli is a classical problem in algebraic geometry closely related to deformation theory. Let $X\subset \PP^n$ be a smooth hypersurface of degree $d$, let $(\PP^n)^*$ denote  the dual projective space, and let $X^*$ be the dual of $X$. We denote by $\mathcal{M}_d$ the moduli space of hypersurfaces of degree $d$ in $\PP^{n-1}$. 

A fundamental  approach for studying  $X$ is to intersect it with hyperplanes  and look at the family of  varieties  thus obtained. This leads to  the morphism
$$
s_X: (\PP^n)^*\setminus X^*\longrightarrow \mathcal{M}_d\cong \PP(H^0(\mathcal{O} _{\PP^{n-1}}(d)))/SL_n \quad \quad H\mapsto X\cap H.$$

The basic properties of $s_X$ are still far from being fully understood. In particular,  the problem of determining  the dimension of its image, known as the  {\em  maximal variation problem for the hyperplane sections}, remains open in general. For plane curves, however, the situation is better understood. Let $C\subset \PP^2$ a plane curve. If $s_C$ fails to have maximal rank then $C$ has infinitely many projective automorphisms \cite[Theorem 1.2]{PRT}. 

From now on we assume $n\ge 3$. In \cite{beauville}, Beauville established a connection between the maximal variation for  hyperplane sections  of a smooth hypersurface $X\subset \PP^n$  and the WLP of its Jacobian ring, $R/J(X)$. 
 As an immediate application of Theorem \ref{arb num var} and using Beauville's criterion (Proposition \ref{criterion}),  we obtain a complete solution to the maximal variation problem for  hyperplane sections  of smooth hypersurfaces of degree $d\ne 3$.

\begin{corollary}
Let \(X \subset \mathbb{P}^n\) be a smooth hypersurface of degree \(d\). Then the hyperplane sections of \(X\) have maximal variation whenever \(d \geq 4\).  
\end{corollary}


\begin{thebibliography}{99}

\bibitem{AR} A. Alzati and R. Re, {\it Complete intersections of quarics and the Weak Lefschetz Property},  Collect. Math. 70 (2019), no. 2, 283--294.

\bibitem{ACGH} E. Arbarello, M. Cornalba, P. Griffiths and J. Harris, ``Geometry of Algebraic Curves, volume~1," Grundlehren der mathematischen Wissenschaften 267 (1985), Springer.

\bibitem{beauville} A. Beauville, {\it Hyperplane sections of cubic threefolds}, Proc. Amer. Math. Soc. 153 (2025), 5167--5170.

\bibitem{beauville2} A. Beauville, {\it Maximal variation of linear systems}, Selecta Mathematica (to appear); available at arXiv 2511.22329.

\bibitem{B-MR} V. Beorchia and R.M. Mir\'o-Roig, {\it Weak Lefschetz property of equigenerated complete intersections. Applications}, Proc. Amer. Math. Soc. 154 (2026), 905--912.

\bibitem{BMMN-codim4} M. Boij, J. Migliore, R.M. Mir\'o-Roig and U. Nagel, {\it On the weak Lefschetz property for height 4 equigenerated complete intersections}, Trans. Amer. Math. Soc., Series B, Vol. 10 (2023), 1254--1286.

\bibitem{BM4} G. Bolondi and J. Migliore, {\it The structure of an even liaison class}, Trans. Amer. Math. Soc. 316 (1989), 1--37.

\bibitem{B-P} D. Bricalli and G. P. Pirola, {\it On the maximal variation problem and Lefschetz pencils}, Preprint arXiv 2606.27893

\bibitem{POLITUS} L. Chiantini, \L. Farnik, G. Favacchio, B. Harbourne, J. Migliore, T. Szemberg and J. Szpond, {\it Intersection of curves in $\PP^4$}, in preparation.

\bibitem{EHbook} D. Eisenbud and J. Harris, ``The practice of algebraic curves  -- A Second Course in Algebraic Geometry,"  Graduate Studies in Mathematics 250 (2024), American Mathematical Society.

\bibitem{HMNW} T.~Harima, J.~C. Migliore, U.~Nagel, and J.~Watanabe. \newblock {\it The weak and strong {L}efschetz properties for {A}rtinian {$K$}-algebras.}  \newblock { J. Algebra}, 262(1):99--126, 2003.

\bibitem{HarrisDuke} J. Harris, {\it Galois groups of ennumerative problems}, Duke Mathematical Journal 46 (1979), 685--724.

\bibitem{harris} J. Harris, {\it The genus of space curves}, Math. Ann. 249 (1980), 191--204.

\bibitem{HU} C.~Huneke and B.~Ulrich.  \newblock {\it  General hyperplane sections of algebraic varieties.}  \newblock  J. Algebraic Geom.  2(3):487--505, 1993.

\bibitem{I} G. Ilardi. \newblock {\it Jacobian ideals, arrangements and the Lefschetz properties.} \newblock { J. Algebra}, 508:418--430, 2018.

\bibitem{kleiman} S.~L. Kleiman.  \newblock {\it Bertini and his two fundamental theorems.}  \newblock { Rend. Circ. Mat. Palermo (2) Suppl.}, (55):9--37, 1998.  \newblock Studies in the history of modern mathematics, III.

\bibitem{KMMNP} J. Kleppe, J. Migliore, R. Mir\'o-Roig, U. Nagel and C. Peterson, ``Gorenstein liaison, complete intersection liaison invariants and unobstructedness," Memoirs Amer. Math. Soc. {\bf 154} (2001), no.\ 732. 

\bibitem{K} M. Kreuzer, {\em On 0-dimensional complete intersections}, Math. Ann. {\bf 292} (1992), 43--58. 

\bibitem{LR} R. Lazarsfeld and A.P. Rao, {\em Linkage of general curves of large degree}," Algebraic geometry -- open problems (Ravello, 1982), 267--289, 
Lecture Notes in Math., {\bf 997}, Springer, Berlin, 1983. 

\bibitem{JM-book} J. Migliore,  ``Introduction to liaison theory and deficiency modules", Birkh\"{a}user, Progress in Mathematics 165, 1998.

\bibitem{PRT} A.~P. Patel, E. Riedl and D. Tseng, {\em Moduli of linear slices of high degree smooth hypersurfaces}, Algebra Number Theory {\bf 18} (2024), no.~12, 2133--2156. 

\bibitem{Rathmann} J. Rathmann, {\it The Uniform Position Principle for Curves in Characteristic $p$},  Math. Ann. 276 (1987) 565--579.

\bibitem{RRR} L.~Reid, L.~G. Roberts, and M.~Roitman.  \newblock {\it On complete intersections and their {H}ilbert functions}.  \newblock{ Canad. Math. Bull.}, 34(4):525--535, 1991.

\bibitem{S} R. P. Stanley, {\em Weyl groups, the hard Lefschetz theorem, and the Sperner property}. SIAM
J. Algebraic Discrete Methods, {\bf 1} (1980), 168?184.

\bibitem{W} J. Watanabe, {\em The Dilworth number of Artinian rings and finite posets with rank function.}
In Commutative algebra and combinatorics (Kyoto, 1985), volume 11 of Adv. Stud. Pure Math., pages
303?312. North-Holland, Amsterdam, 1987.

 \end{thebibliography}
\end{document}